\documentclass[11pt]{amsart}

\title{The Long-Time Behavior Of The Ricci Tensor Under The Ricci Flow}
\author{Christian Hilaire}
\date{}

\usepackage{amsmath,amsthm,amssymb}

\newtheorem{thm}{Theorem}

\newtheorem{lma}[thm]{Lemma}

\newtheorem{crly}[thm]{Corollary}

\begin{document}

\maketitle

\begin{abstract}
We show that, given an immortal solution to the Ricci flow on a closed manifold with uniformly bounded curvature and diameter, the Ricci tensor goes to zero as $t \rightarrow \infty$. We also show that if there exists an immortal solution on a closed 3-dimensional manifold such that the product $diam\left(M;g(t)\right)^2||Rm||_{\infty}(t)$ is uniformly bounded, then this solution must be of type III.
\end{abstract}
 
\section{Introduction}
  The Ricci flow equation, introduced by R. Hamilton in \cite{hamilton1}, is the nonlinear partial differential equation
  \begin{align}
  \frac{\partial}{\partial t}g(t)=-2Ric\left(g(t)\right) \label{ricciflow}
  \end{align}
  where $g(t)$ is a Riemannian metric on a fixed smooth manifold $M$. Hamilton showed that, given any Riemannian metric $g_0$ on a closed manifold $M$, there exists $T>0$ such that the equation (\ref{ricciflow}) has a solution $g(t)$ defined for $t \in [0,T)$ and which satisfies $g(0)=g_0$. A solution to the Ricci flow equation which is defined for all $t \geq 0$ is called an immortal solution. If the solution is defined for all $t \in \mathbb{R}$, it is said to be eternal.
   
   In this paper, we first consider immortal solutions to the Ricci flow which have a uniform bound on the curvature and a uniform upper bound on the diameter. We will use the following notation. Let $M$ be an $n$-dimensional closed manifold and let $g(t)$ be an immortal solution to the Ricci flow on $M$. For each $t \geq 0$, we set 
 \begin{align*} 
 ||Rm||_{\infty}(t)=\max\{|Rm|(x,t) \, : \, x \in M \}
 \end{align*}
 and
 \begin{align*} 
 ||Ric||_{\infty}(t)=\max\{|Ric|(x,t) \, : \, x \in M \}.
 \end{align*}
 The diameter of the induced metric structure on $M$ will be denoted by $diam\left(M;g(t)\right)$ for each $t \geq 0$.  The following theorem basically states that given an immortal solution with a uniform bound on the curvature and on the diameter, the Ricci tensor must tend to zero in a uniform way as $t \rightarrow \infty$.

\begin{thm}
\label{big1}
Given $n \in \mathbb{N}$ and numbers $K,D >0$, there exists a function $F_{n,K,D}:[0,\infty) \rightarrow [0,\infty)$  with the following properties
\begin{enumerate}
\item  $\displaystyle \lim_{t \rightarrow \infty}F_{n,K,D}(t)=0$.
 \item If $g(t)$ is an immortal solution to the Ricci flow on a closed manifold $M$ of dimension $n$ such that $||Rm||_{\infty}(t) \leq K$ and $diam\left(M;g(t)\right) \leq D$ for all $t \geq 0$, then $||Ric||_{\infty}(t) \leq F_{n,K,D}(t)$ for all $t \geq 0$.
 \end{enumerate}
\end{thm}

\begin{crly}
 Suppose $g(t)$ is an immortal solution to the Ricci flow on a closed $n$-dimensional manifold $M$ and suppose that, for some constants $K, D >0$, we have $||Rm||_{\infty}(t) \leq K$ and $diam\left(M;g(t)\right) \leq D$ for every $t \geq 0$. Then $||Ric||_{\infty}(t) \rightarrow 0$ as $t \rightarrow \infty$.
 \end{crly}
 
 The next theorem is about the type of a certain class of immortal solutions on three-dimensional closed manifolds. Recall that an immortal solution $g(t)$ to the Ricci flow on a closed manifold $M$ is said to be of type III if $ \displaystyle sup_{M\times [0,\infty)}t|Rm|<\infty$.

\begin{thm}
\label{bigdos}
 Suppose $g(t)$ is an immortal solution to the Ricci flow on a closed 3-dimensional manifold $M$ and suppose there exists a constant $C>0$ such that 
 \begin{align}
 diam\left(M;g(t)\right)^2||Rm||_{\infty}(t) < C \label{kontrol}
 \end{align}
 for all $t \geq 0$. Then, the solution $g(t)$ is a type III solution.
 \end{thm}
  It is still an open question whether an immortal solution on a closed 3-dimensional manifold is nececessarily of type III.

  \textbf{Acknowledgements:} I would like to thank my adviser John Lott for helpful discussions and comments on earlier drafts of this paper.
  
 \section{ The proof}
 We first prove Theorem~\ref{big1} by assuming in addition a uniform lower bound on the injectivity radius. The proof of this simpler case will basically give us an outline of the proof for the general case. Given a Ricci flow solution $g(t)$ on a closed manifold $M$, the injectivity radius corresponding to each metric $g(t)$ will be denoted by $inj\left(M;g(t)\right)$. 
 
 \begin{thm}
 \label{big2}
Given $n \in \mathbb{N}$ and numbers $K,D, \iota>0$, there exists a function $F_{n,K,D,\iota}:[0,\infty) \rightarrow [0,\infty)$ with the following properties
\begin{enumerate}
\item $\displaystyle \lim_{t \rightarrow \infty}F_{n,K,D,\iota}(t)=0$.
\item If $g(t)$ is an immortal solution to the Ricci flow on a closed manifold $M$ of dimension $n$ such that $inj(M;g(t)) \geq \iota$,
 $||Rm||_{\infty}(t) \leq K$ and  $diam\left(M;g(t)\right) \leq D$ for all $t \geq 0$, then $||Ric||_{\infty}(t) \leq F_{n,K,D,\iota}(t)$ for all $t \geq 0$.
\end{enumerate}
\end{thm}

\begin{proof}
 For every $t \geq 0$, we set $F_{n,K,D,\iota}(t)$ to be the supremum of $||Ric||_{\infty}(t)$ over all immortal solutions $\left(M^n,g(.)\right)$ satisfying the conditions of the given hypothesis. This function is well defined since the bound on the curvature gives us a bound on the Ricci curvature. We just need to prove that $\displaystyle \lim_{t \rightarrow \infty}F_{n,K,D,\iota}(t)=0$.
 
 Suppsose that this did not hold. Then, for some $\epsilon >0 $, there is a sequence $t_{i} \rightarrow \infty$  such that $F_{n,K,D,\iota}(t_i)> \epsilon$. This implies that there exists a sequence $\left(M,g_{i},x_i \right)$ of pointed immortal solutions to the Ricci flow such that $|Ric|(x_{i},t_{i})>\epsilon $. Consider the new sequence $\left(M,\tilde{g}_{i},x_i \right)$ where $\tilde{g}_i(t)=g_i(t+t_i)$ for $t \in [-t_i,\infty)$. Since we have a uniform bound on the curvature and a uniform lower bound on the injectivity radius, we can apply the Hamilton-Cheeger-Gromov compactness theorem \cite{hamilton}. After passing to a subsequence, we can assume that the sequence $\left(M_i,\tilde{g}_i(t),x_i\right)$ converges to an eternal Ricci flow solution $\hat{g}(t)$ on a pointed $n$-dimensional manifold $(\hat{M},\hat{x})$. The uniform bound on the diameters of $\left(M_i,\tilde{g}_{i}\right)$ imply that $\hat{M}$ is a closed manifold.
   
    The eternal Ricci flow solution $\hat{g}(t)$ must be Ricci flat. Indeed, based on \cite[ Lemma 2.18]{chow}, the solution is either Ricci flat or its scalar curvature is positive. Since $\hat{M}$ is closed, the latter cannot happen for, by \cite[Corollary 2.16]{chow}, there would be a finite time singularity and thus would contradict the fact that $\hat{g}(t)$ is an eternal solution. Therefore, the Ricci flow solution $\hat{g}$ satisfies $Ric\left(\hat{g}(t)\right)\equiv 0$ as claimed. On the other hand, the choice of the points $x_i \in M_i$ implies that $|Ric|(\hat{x},0) > \epsilon$. We reach a contradiction. The function $F_{n,K,D,\iota}(t)$ must satisfy $\displaystyle \lim_{t \rightarrow \infty}F_{n,K,D,\iota}(t)=0$. This completes the proof.
 \end{proof}

 The proof of the previous Theorem is a simple application of the following results.
 \begin{enumerate}
 \item Hamilton's compactness theorem for solutions to the Ricci flow.
 \item An eternal solution to the Ricci flow on a closed manifold is Ricci flat. This result is a simple application of the weak and strong maximum principles.
 \end{enumerate}
 In the proof of Theorem~\ref{big1}, we will use, in the bounded curvature case, a more general version of Hamilton's compactness theorem which was introduced by J. Lott in \cite{lott}. We will again construct a pointed sequence of Ricci flow solutions. But in this case, after passing to a subsequence if necessary, the pointed sequence will converge, in the sense of \cite[Section 5]{lott}, to an eternal Ricci flow solution on a pointed \'etale groupoid with compact connected orbit space. The result will then follow by a simple application of the weak and strong maximum principles in the case of \'etale groupoids. For basic informations about groupoids, we refer to \cite{bridson,moerdijk}. For information about Riemannian groupoids and their applications in Ricci flow, we refer to \cite{lott}. We will denote an \'etale groupoid by a pair $(\mathcal{G},M)$ where $\mathcal{G}$ is the space of ``arrows'' and $M$ is the space of units. The source and range maps will be denoted by $s: \mathcal{G} \rightarrow M$ and $r: \mathcal{G} \rightarrow M$ respectively.
    
 \begin{lma}[Weak maximum principle for scalars on groupoids]\label{weakmaxlem}
 Suppose $g(t)$, $ 0 \leq t \leq T < \infty$, is a smooth one-paramater family of metrics on a smooth \'etale groupoid $(\mathcal{G},M)$ with compact connected orbit space $W=M/\mathcal{G}$. Let $X(t)$ be a smooth $\mathcal{G}$-invariant time-dependent vector field on $M$ and let $F:\mathbb{R}\times [0,T] \rightarrow \mathbb{R}$ be a smooth function. Suppose that $u:M\times [0,T] \rightarrow \mathbb{R}$ is a smooth $\mathcal{G}$-invariant function which solves
\begin{align}
\frac{\partial u}{\partial t} \leq \Delta_{g(t)}u + < X(t),\nabla u > + F(u,t). \label{weakmax1}
\end{align}
Suppose further that $\phi: [0,T] \rightarrow \mathbb{R}$ solves
\begin{align}
\begin{cases}
 \frac{d\phi}{dt}&=F(\phi(t),t) \label{weakmax2} \\  
\phi(0)&=\alpha \in \mathbb{R}
\end{cases}
\end{align}

If $u(.,0)\leq \alpha$, then $u(.,t) \leq \phi(t)$ for all $t \in [0,T]$.
\end{lma}

\begin{proof}
 The proof is basically the same as the corresponding proof for closed manifolds (see\cite[p. 35-36]{topping}). For $\epsilon >0$, we consider the ODE 
\begin{equation}
\begin{cases}
\frac{d\phi_{\epsilon}}{dt}&=F(\phi_{\epsilon}(t),t) + \epsilon \label{weakmax3}\\  
\phi_{\epsilon}(0)&=\alpha +\epsilon \in \mathbb{R},
\end{cases}
\end{equation}
for a new function $\phi_{\epsilon}:[0.T] \rightarrow \mathbb{R}$. Using basic ODE theory, we see that the existence of $\phi$ asserted in the hypotheses and the fact that $T < \infty$ imply that for some $\epsilon_{0}>0$  there exists a solution $\phi_{\epsilon}$ on $[0,T]$ for any  $0<\epsilon \leq \epsilon_{0}$. Furthermore, $\phi_{\epsilon} \rightarrow \phi$ uniformly as $\epsilon \rightarrow 0$. Thus, it suffices to show that $u(.,t) < \phi_{\epsilon}(t)$ for all $t \in [0,T]$ and arbitrary $\epsilon \in (0,\epsilon_{0})$.

If this were not true, then we could choose $\epsilon \in (0,\epsilon_{0})$ and $t_{0} \in (0,T]$ where $u(.,t_{0})< \phi_{\epsilon}(t_0)$ fails. We may assume that $t_0$ is the earliest such time, and pick $x \in M$ such that $u(x,t_0)=\phi_{\epsilon}(t_0)$. Indeed, such a pair $(x,t_0)$ exists because the $\mathcal{G}$-invariant function $u$ descends to a continous function $\tilde{u}:W \times [0,T] \rightarrow \mathbb{R}$. Based on the hypothesis that $W$ is compact and connected, we can find $t_0$ such that it is the earliest time where $ \tilde{u}(.,t_{0})< \phi_{\epsilon}(t_0)$ fails and pick an orbit $O_{x}$ so that $\tilde{u}(O_{x},t_0)=\phi_{\epsilon}(t_0)$. This gives the required pair $(x,t_0)$. Now, we know that $u(x,s)-\phi_{\epsilon}(s)$ is negative for $s \in [0,t_0)$ and zero for $s=t_0$. Hence, we must have 
\begin{align*}
\frac{\partial u}{\partial t}(x,t_0)-\phi_{\epsilon}^{\prime}(t_0)\geq 0.
\end{align*}
Moreover, since $u(.,t_0)$ achieves a maximum at $x$, we have $\nabla u(x,t_0)=0$ and $\Delta u(x,t_0) \leq 0$.

Combining these facts with the inequality (\ref{weakmax1}) for $u$ and equation (\ref{weakmax3}) for $\phi_{\epsilon}$, we get the contradiction  
\begin{align*}
0 &\geq \left[\frac{\partial u}{\partial t} -\Delta_{g(t)}u -< X,\nabla u > -F(u,.)\right](x,t_0)\\
\Rightarrow 0   &\geq \phi^{\prime}_{\epsilon}(t_0)-F(\phi_{\epsilon}(t_0),t_0)=\epsilon>0.
\end{align*} 
This completes the proof.
\end{proof}

If $g(t)$ is a solution to the Ricci flow on a smooth \'etale groupoid $(\mathcal{G},M)$, then, by definition, the elements of $\mathcal{G}$ act by isometries for each Riemannian metric $g(t)$. More precisely, recall that for any \'etale groupoid $(\mathcal{G},M)$, there is an associated pseudogroup of local diffeomorphisms of $M$. The elements of this pseudogroup are diffeomorphisms $\phi : U \rightarrow V$ where $U$ and $V$ are open subsets of $M$ and $\phi$ is of the form $\phi=r\circ \alpha$  where $\alpha:U \rightarrow \mathcal{G}$ is a local section of the source map $s$, that is $\alpha$ satifies $s \circ \alpha = Id$. Given a Riemmanian metric $g$ on $M$, the elements of $\mathcal{G}$ are said to act by isometries if the elements of the associated pseudogroup are local isometries. Therefore, given a solution to the Ricci flow on a smooth \'etale groupoid, the geometric quantities derived from the metrics $g(t)$ are $\mathcal{G}$-invariant, that is they are invariant under the action of the associated pseudogroup. It follows that the basic evolution equations for the Ricci flow which are derived by applying the weak maximum principle on a closed manifold are still valid in the case of the Ricci flow on a smooth \'etale groupoid with compact connected orbit space. 

\begin{lma}
\label{main} 
Every eternal solution to the Ricci flow on a smooth \'etale groupoid with compact connected orbit space is Ricci flat.
\end{lma}
\begin{proof}
Suppose $(\mathcal{G},M,g(t))$ is an eternal solution to the Ricci flow on a smooth \'etale groupoid with compact connected orbit space $W$. Just as in the closed manifold case ( see \cite[Corollary 3.2.5]{topping}), one can apply Lemma \ref{weakmaxlem} to show that for any solution to the Ricci flow defined on an interval $[0,T]$ the scalar curvature satisfies 
 \begin{align}
 R \geq -\frac{n}{2t} \label{menli}
 \end{align}
 for $t \in [0,T]$. Fix $c \in \mathbb{R}$. By translating time, we see that the above inequality implies that
 \begin{align*}
 R(x,t) \geq -\frac{n}{2(t-c)}
 \end{align*}
 for $(x,t) \in M \times (c,\infty)$. Taking the limit as $c \rightarrow -\infty$, we conclude that $R(x,t) \geq 0$ for all $(x,t) \in M \times \mathbb{R}$.
 
 Now suppose that $R(x_0,t_0) >0$ for some $(x,t_0) \in M \times \mathbb{R}$. By the strong maximum principle (\cite[Corollary 6.55]{chow}), this implies that $R(x,t) >0$ for all $t \geq t_0$ and for all $x$ in the connected component of $M$ which contains $x_0$.  In fact, this holds for all $x \in M$. Indeed, since $W$ is connected, it follows that for any $x \in M$, one can find a sequence of points $x_0=p_0,q_0,p_1,q_1,\cdots,p_n,q_n=x$ such that $p_i$ and $q_i$ are in the same connected component of $M$, and $q_{i-1}$ and $p_i$ are in the same $\mathcal{G}$-orbit for each $i$. Since $R$ is $\mathcal{G}$-invariant, it follows that $R(p_1,t_0)=R(q_0,t_0)>0$. So, by the maximum principle, $R(x,t)>0$ for every $t \geq t_0$ and for all $x$ in the connected component of $M$ which contains $p_1$. In particular, $R(q_1,t)>0$ for all $t \geq t_0$. After applying this argument a finite number of times, we deduce that $R(x,t) >0$ for all $t \geq t_0$. This proves the claim.
  
 The $\mathcal{G}$-invariant function $R$ induces a continuous function $\tilde{R}$ on $W$ which satifsies $\tilde{R}(.,t)>0$ for $t \geq t_0$. Since $W$ is compact, there exists some $\alpha >0$ such that $\tilde{R}(.,t_0) \geq \alpha$. Hence the scalar curvature satisfies $R(.,t_0) \geq \alpha$. Just as  in the closed manifold case (see \cite[Corollary 3.2.3]{topping}), we can apply Lemma \ref{weakmaxlem} to show that we have a finite time singularity. This contradicts the hypothesis that the Ricci flow solution is eternal. Therefore, we must have $R \equiv 0$. Based on the evolution equation of the scalar curvature
 \begin{align*}
 \frac{\partial R}{\partial t} = \Delta R + 2|Ric|^2
 \end{align*}
 we deduce that $Ric \equiv 0$. This completes the proof.
 \end{proof}

 \begin{proof}
  (Theorem 1).We define the function $F_{n,K,D}(t)$ as before: for every $t \geq 0$, we set $F_{n,K,D}(t)$ to be the supremum of $||Ric||_{\infty}(t)$ over all immortal solutions $\left(M^n,g(.)\right)$ satisfying the conditions of the given hypothesis.We just need to prove that $\displaystyle \lim_{t \rightarrow \infty}F_{n,K,D}(t)=0$. This will again be proved by contradiction.
 
 Suppose that this did not hold. Then, for some $\epsilon >0 $, there is a sequence $t_{i} \rightarrow \infty$  such that $F_{n,K,D}(t_i)> \epsilon$. This implies that there exists a sequence $\left(M,g_{i},x_i \right)$ of pointed immortal solutions to the Ricci flow such that $|Ric|(x_{i},t_{i})>\epsilon $. Consider the new sequence $\left(M,\tilde{g}_{i},x_i \right)$ where $\tilde{g}_i(t)=g_i(t+t_i)$ for $t \in [-t_i,\infty)$. After passing to a subsequence, we can assume that the sequence $\left(M_i,\tilde{g}_i(t),x_i\right)$ converges to an eternal Ricci flow solution $\hat{g}(t)$ on a pointed $n$-dimensional \'etale groupoid $(\hat{\mathcal{G}},\hat{M},O_{\hat{x}})$ (see \cite[Theorem 5.12]{lott}). The uniform bound on the diameters of $\left(M_i,\tilde{g}_{i}\right)$ imply that the quotient space $W$ equipped with the metric induced by $g(t)$ is a locally compact complete length space with finite diameter. So, by the Hopf-Rinow theorem (\cite[ Part I Proposition 3.7]{bridson}), $W$ is a compact connected space. We can now apply Lemma~\ref{main} to deduce that the eternal Ricci flow solution on the \'etale groupoid $(\hat{\mathcal{G}},\hat{M},O_{\hat{x}})$ is Ricci flat. On the other hand, the choice of the points $x_i \in M_i$ implies that $|Ric|(.,0) > \epsilon$ on $O_{\hat{x}}$. We reach a contradiction. The function $F_{n,K,D,}(t)$ must satisfy $\displaystyle \lim_{t \rightarrow \infty}F_{n,K,D}(t)=0$. This completes the proof.
 \end{proof} 
 
 The proof of Theorem~\ref{bigdos} will also follow from Lemma~\ref{main}
 \begin{proof}
 (Theorem 3).The proof will again be by contradiction. Suppose the solution  $g(t)$ satisfies $ \displaystyle sup_{M\times [0,\infty)}t|Rm|=\infty$. As outlined in \cite[Chapter 8, Section 2]{chow}, we can construct a sequence of pointed solutions to the Ricci flow which converges to an eternal solution to the Ricci flow on an \'etale groupoid. We first choose a sequence of times $T_i \rightarrow \infty$ and then we choose a sequence $(x_i,t_i) \in M \times (0, T_i)$ such that
 \begin{align*}
 t_i(T_i-t_i)|Rm|(x_i,t_i)  = sup_{M \times [0,T_i]} t(T_i-t)|Rm|(x,t)
 \end{align*}
 For each $i$, we set $K_i = |Rm|(x_i,t_i)$ and we define the pointed Ricci flow solution $(M,g_i,x_i)$ by $g_i(t)=K_ig(t_i+\frac{t}{K_i})$. If we set $\alpha_i=-t_i K_i$ and $\omega_i=(T_i-t_i)K_i$, then it is shown in \cite[Chapter 8, Section 2]{chow} that $\alpha_i \rightarrow -\infty$ and $\omega_i \rightarrow \infty$, and that the curvature $Rm(g_i)$ of the metric $g_i$ satisfies 
 \begin{align*}
 |Rm(g_i)|(x,t) \leq \frac{\alpha_i}{\alpha_i-t}\frac{\omega_i}{\omega_i-t}
 \end{align*}
  for $(x,t) \in M \times [\alpha_i,\omega_i]$. Furthermore, the condition (\ref{kontrol}) implies that 
  \begin{align*}
  diam(M;g_i(0))^2 = K_i diam(M;g(t_i))^2 < C
  \end{align*}
  Hence, after passing to a subsequence if necessary, the pointed sequence $(M,g_i,x_i)$ converges to an eternal Ricci flow solution $\hat{g}(t)$ on a pointed $3$-dimensional \'etale groupoid $(\hat{\mathcal{G}},\hat{M},O_{\hat{x}})$ with compact connected orbit space. By construction, $|Rm(\hat{g})|(.,0)=1$ on the orbit $O_{\hat{x}}$. But, by Lemma~\ref{main}, the solution $\hat{g}(t)$ is Ricci flat and, hence, flat since we are in the 3-dimensional case. We get a contradiction. Therefore, the immortal solution $g(t)$ must be of type III.
  \end{proof}

 \address{University Of California, Berkeley, 840 Evans Hall, Berkeley, CA, 94720
 \emph{E-mail address}: \email{hilaire@math.berkeley.edu}
 \end{document}